\documentclass{amsart}
\usepackage[all]{xy}
\xyoption{poly}
\xyoption{arc}
\xyoption{curve}

\CompileMatrices

\begin{document}

\newtheorem{theorem}{Theorem}{}
\newtheorem{lemma}[theorem]{Lemma}{}
\newtheorem{corollary}[theorem]{Corollary}{}
\newtheorem{conjecture}[theorem]{Conjecture}{}
\newtheorem{proposition}[theorem]{Proposition}{}
\newtheorem{axiom}{Axiom}{}
\newtheorem{remark}{Remark}{}
\newtheorem{example}{Example}{}
\newtheorem{exercise}{Exercise}{}
\newtheorem{definition}{Definition}{}
\newcommand{\nc}{\newcommand}

\nc{\gr}{\bullet}
\nc{\pr}{\noindent{\em Proof. }} \nc{\g}{\mathfrak g}
\nc{\n}{\mathfrak n} \nc{\opn}{\overline{\n}}\nc{\h}{\mathfrak h}
\renewcommand{\b}{\mathfrak b}
\nc{\Ug}{U(\g)} \nc{\Uh}{U(\h)} \nc{\Un}{U(\n)}
\nc{\Uopn}{U(\opn)}\nc{\Ub}{U(\b)} \nc{\p}{\mathfrak p}
\renewcommand{\l}{\mathfrak l}
\nc{\z}{\mathfrak z} \renewcommand{\h}{\mathfrak h}
\nc{\m}{\mathfrak m}
\renewcommand{\t}{\mathfrak t}
\renewcommand{\k}{\mathfrak k}
\nc{\opk}{\overline{\k}}
\nc{\opb}{\overline{\b}}

\title{The structure of the nilpotent cone, \\ the Kazhdan--Lusztig map and algebraic
group analogues of the Slodowy slices}

\author{A. Sevostyanov}
\address{ Institute of Pure and Applied Mathematics,
University of Aberdeen \\ Aberdeen AB24 3UE, United Kingdom \\ email: a.sevastyanov@abdn.ac.uk}

\begin{abstract}
We define algebraic group analogues of the Slodowy transversal
slices to adjoint orbits in a complex semisimple Lie algebra $\g$.
The new slices are transversal to the conjugacy classes in an
algebraic group $G$ with Lie algebra $\g$. These slices are associated
to (the conjugacy classes of) elements $s$ of the Weyl group $W$ of $\g$. For such
slices we prove an analogue of the Kostant cross--section theorem
for the action of a unipotent group.
\end{abstract}


\maketitle

\pagestyle{myheadings}

\markboth{ALEXEY SEVOSTYANOV}{ALGEBRAIC GROUP ANALOGUES OF SLODOWY SLICES}

\renewcommand{\theequation}{\thesection.\arabic{equation}}

\section{Introduction}

\setcounter{equation}{0}

Let $\g$ be a complex semisimple Lie algebra, $G$ the adjoint
group of $\g$, $e\in \g$ a nonzero nilpotent element in $\g$. By
the Jacobson--Morozov theorem there is an
$\mathfrak{sl}_2$--triple $(e,h,f)$ associated to $e$, i.e.
elements $f,h\in \g$ such that $[h,e]=2e$, $[h,f]=-2f$, $[e,f]=h$.
Fix such an $\mathfrak{sl}_2$--triple.

Let $z(f)$ be the centralizer of $f$ in $\g$. The affine space
$s(e)=e+z(f)$ is called the Slodowy slice to the adjoint orbit of
$e$ at point $e$. Slodowy slices were introduced in \cite{SL} as a
technical tool for the study of the singularities of the adjoint
quotient of $\g$. We recall that if $\h$ is a Cartan subalgebra in
$\g$ and $W$ is the Weyl group of $\g$ then, after identification
$\g\simeq \g^*$ with the help of the Killing form, the adjoint
quotient can be defined as the morphism $\delta_\g: \g \rightarrow
\h/W$ induced by the inclusion $\mathbb{C}[\h]^W\simeq
\mathbb{C}[\g]^G\hookrightarrow \mathbb{C}[\g]$. The fibers of
$\delta_\g$ are unions of adjoint orbits in $\g$. Each fiber of
$\delta_\g$ contains a single orbit which consists of regular
elements. The singularities of the fibers correspond to irregular
elements.

Slodowy studied the singularities of the adjoint quotient by
restricting the morphism $\delta_\g$ to the slices $s(e)$ which turn
out to be transversal to the adjoint orbits in $\g$. In
particular, for regular nilpotent $e$ the restriction
$\delta_\g:s(e)\rightarrow \h/W$ is an isomorphism, and $s(e)$ is a
cross--section for the set of the adjoint orbits of regular
elements in $\g$. For subregular $e$ the fiber $\delta_\g^{-1}(0)$
has one singular point which is a simple singularity, and $s(e)$
can be regarded as a deformation of this singularity.

In this paper we are going to outline a similar construction for algebraic groups.
Let $G$ be a complex simple algebraic group with Lie algebra $\g$.
In case of algebraic groups, instead of the adjoint quotient map, one should consider the conjugation quotient map
$\delta_G: G \rightarrow H/W$ generated by the inclusion $\mathbb{C}[H]^W\simeq
\mathbb{C}[G]^G\hookrightarrow \mathbb{C}[G]$, where $H$ is the maximal torus of $G$ corresponding to the Cartan subalgebra
$\h$ and $W$ is the Weyl group of the pair $(G, H)$. Some fibers of this map are singular and one can study these singularities
by restricting $\delta_G$ to certain transversal slices to conjugacy classes in $G$. We are going to define such slices in this paper.

Note that the fibers of the adjoint quotient map and those of the conjugation quotient map are generally not isomorphic.
More precisely, by Theorem 3.15 in \cite{SL} there are open neighborhoods $U$ of $1$ in $H/W$ and $U'$ of $0$ in
$\h/W$ and a surjective morphism $\gamma:U\rightarrow U'$ such that the fibers $\delta_G^{-1}(u)$ and
$\delta_{\g}^{-1}(\gamma(u))$ are isomorphic for $u\in U$ as $G$--spaces. But globally such isomorphisms do not exist.
In fact any fiber of the adjoint quotient map can be translated by a contracting $\mathbb{C}^*$--action to a fiber over any
small neighborhood of $0$ in $\h/W$. But there is no similar statement for the fibers of the conjugation quotient map,
and the problem of the study of singularities of the fibers of the conjugation quotient map and of their resolutions and
deformations is more difficult than the same problem in case of the adjoint quotient map.

The other construction where the Slodowy slices or, more precisely, noncommutative deformations of
algebras of regular functions on these slices, play an important
role is the Whittaker or, more generally, generalized
Gelfand--Graev representations of the Lie algebra $\g$ (see
\cite{Ka,K}). Namely, to each nilpotent element in $\g$ one can associate the corresponding category of Gelfand--Graev representations.
The category of
generalized Gelfand--Graev representations associated to a nilpotent element $e\in \g$ is equivalent to the
category of finitely generated left modules over the W--algebra associated to $e$. This
remarkable result was proved by Kostant in case of regular
nilpotent $e\in \g$ (see \cite{K}) and by Skryabin in the general case
(see Appendix to \cite{Pr}). A more direct proof of Skryabin's
theorem was obtained in \cite{GG}. This proof, as well as the
original Kostant's proof, is based on the study of the commutative
graded algebra associated to the algebra W--algebra.

The main observation of \cite{GG,K} is that this commutative algebra is isomorphic to the
algebra of regular functions on the Slodowy slice $s(e)$. This isomorphism is established with the help of
a cross--section theorem proved in \cite{K} in case of regular
nilpotent $e$ and in \cite{DB,GG} in the general case. We briefly recall the main statement of this theorem.

Let $\chi$ be the element of
$\g^*$ which corresponds to $e$ under the isomorphism $\g\simeq
\g^*$ induced by the Killing form. Under the action of ${\rm ad}~h$
we have a decomposition
\begin{equation}\label{deco}
\g=\oplus_{i\in \mathbb{Z}}\g(i),~{\rm where}~ \g(i)=\{x\in \g
\mid [h,x]=ix\}.
\end{equation}
The skew--symmetric bilinear form $\omega$ on $\g(-1)$ defined by
$\omega(x,y)=\chi([x,y])$ is nondegenerate. Fix an isotropic
subspace $l$ of $\g(-1)$ with respect to $\omega$ and denote by
${l}^{\perp_\omega}$ the annihilator of ${l}$ with respect to
$\omega$.

Let
\begin{equation}\label{nopno}
\m_l=l\oplus \bigoplus_{i\leq -2}\g(i),~~\n_l
={l}^{\perp_\omega}\oplus \bigoplus_{i\leq -2}\g(i).
\end{equation}
Note that $\m_l\subset \n_l$, both $\m_l$ and $\n_l$ are nilpotent
Lie subalgebras of $\g$.

The Kostant cross--section theorem
asserts that the adjoint action map $N_l\times s(e)\rightarrow
e+\m_l^{\perp_\g}$ is an isomorphism of varieties, where
$N_l$ is the
Lie subgroup of $G$ corresponding to the Lie subalgebra
$\n_l\subset \g$, and $\m_l^{\perp_\g}$ is the annihilator of $\m_l$ in $\g$ with
respect to the Killing form.

Now a natural question is: are there any analogues of the Slodowy
slices for algebraic groups? According to a general theorem proved in \cite{SL}, if $G$ is an algebraic group one can
construct a transversal slice to the set of $G$--orbits at each point of any variety $V$ equipped with a $G$--action. In particular, such
slices exist for $V=G$ equipped with conjugation action.
But we are rather interested in special transversal slices which can be used for the study of
singularities of the conjugation quotient map and for which an analogue of the Kostant cross--section theorem holds.

In paper \cite{St2} R. Steinberg
introduced a natural analogue of the slice $s(e)$ for regular
nilpotent $e$. We recall that $s(e)$ is a cross--section for the
set of adjoint orbits of regular elements in $\g$. In paper
\cite{St2} a cross--section for the set of conjugacy classes of
regular elements in the connected simply connected group $G'$ with
Lie algebra $\g$ is constructed. We briefly recall Steinberg's
construction.

If $e$ is regular nilpotent then, in the notation introduced
above, $\g(-1)=0$, and $\n_l=\m_l=\n$, where $\n$ is a maximal
nilpotent subalgebra of $\g$. Let $\p=\bigoplus_{i\leq 0}\g(i)$ be
the Borel subalgebra containing $\n$, $\h=\g(0)$ the Cartan
subalgebra of $\g$, and $W$ the Weyl group of the pair $(\g,\h)$.
Fix a system of positive simple roots associated to the pair
$(\h,\p)$. Let $s\in W$ be a Coxeter element, i.e. a product of
the reflections corresponding to the simple roots. Fix a
representative for $s$ in $G'$. We denote this representative by
the same letter. Let $N$ be the unipotent subgroup in $G'$
corresponding to the Lie algebra $\n$, and $\overline P$ the
opposite Borel subgroup with Lie algebra $\overline
\p=\bigoplus_{i\geq 0}\g(i)$.

Steinberg introduced a subgroup $N_s\subset N$, $ N_s=\{n\in N\mid
sns^{-1}\in \overline P\}, $ and proved that the set $N_ss^{-1}$
is a cross--section for the set of conjugacy classes of regular
elements in the connected simply connected algebraic group $G'$
associated to the Lie algebra $\g$. Moreover, in \cite{STSh} it is
shown that the conjugation map $N\times N_ss^{-1}\rightarrow
Ns^{-1}N$ is an isomorphism of varieties. The last statement is an
algebraic group analogue of the Kostant cross--section theorem.

As it was observed in \cite{S2,S3}
the analogue of the Kostant cross--section theorem for the slice
$N_ss^{-1}$ is the main ingredient of the construction of the Whittaker model of the center of the quantum group
and of the Whittaker representations for quantum groups.
The purpose of this paper is to construct other special transversal slices to
conjugacy classes in a complex semisimple algebraic group $G$ and
to find an analogue of the Kostant cross--section theorem for
these slices. We expect that these results can be applied to define deformed W--algebras and the generalized
Gelfand--Graev representations for quantum groups. An initial step in this programme was realized in preprint \cite{S}
where the Poisson deformed W--algebras are defined with the help of Poisson reduction in algebraic Poisson--Lie groups.

As we shall see in Section \ref{slices1} transversal
slices in $G$ similar to the Steinberg slice appear in a quite
general setting. They are associated to Weyl group elements $s\in W$. Such a slice always contains
a representative $s^{-1}\in G$ of the element $s^{-1}\in W$. Note that since any Weyl group element
$s\in W$ has finite order its representative $s\in G$ is semisimple,
and hence the slice associated to $s$ is always transversal
to the set of conjugacy classes in $G$ at the semisimple element $s^{-1}\in G$ while
the Slodowy slice associated to a nilpotent element $e$ is always transversal
to the set of adjoint orbits at the nilpotent element $e$.

An important ingredients of our construction are parabolic subgroups $P$ associated to Weyl group elements $s\in W$ in such a way that the semisimple part of the Levi factor of a parabolic subgroup $P$ associated to $s\in W$ is contained in the centralizer of the normal representative $s\in G$ of the Weyl group element $s\in W$.
The parabolic subgroups associated to elements of the Weyl group
play the role of parabolic subalgebras $\p=\bigoplus_{i\leq 0}\g(i)$ associated to nilpotent
elements of $\g$ with the help of grading (\ref{deco}).

The new transversal slices in $G$ are of the form $N_sZs^{-1}$ where $s\in G$ is the normal representative of an element of the Weyl group $W$, $N_s=\{n\in N\mid sns^{-1}\in \overline N\}$, $N$ is the unipotent radical of a parabolic subgroup $P$ corresponding to $s$ with Levi factor $L$, $\overline N$ is the unipotent radical of the opposite parabolic subgroup, and $Z=\{z\in L\mid szs^{-1}=z\}$ is the centralizer of $s$ in $L$. In Proposition \ref{prop2} we prove that the quotient $NZs^{-1}N/N$ with respect to the action of $N$ on $NZs^{-1}N$ by conjugations is isomorphic to $N_sZs^{-1}$, and hence $N_sZs^{-1}$ is a subvariety of $G/N$. This is the algebraic group analogue of the Kostant cross--section theorem.

In Section \ref{sectsub} we apply the results of Section \ref{slices1} to describe simple singularities
in terms of transversal slices in algebraic groups. In our construction we use the subregular slices of dimension ${\rm rank}~G+2$ associated to the elements of the Weyl
group which are related to subregular nilpotent elements in $\g$ via the Kazhdan--Lusztig map (see \cite{KL}).
The intersections of the subregular slices $N_sZs^{-1}$ of dimension ${\rm rank}~G +2$ with the fibers of the conjugation quotient map may only have simple isolated singularities.

{\bf Acknowledgement}

The author is grateful to A. Berenstein, P. Etingof, L. Feh\'{e}r, A. Levin and  A. Premet for useful
discussions. This paper was completed during
my stay in Max--Planck--Institut f\"{u}r Mathematik, Bonn in April--May 2009. I would like to thank Max--Planck--Institut
f\"{u}r Mathematik, Bonn for hospitality.

%%%%%%%%%%%%%%%%%%%%%%%%%%%%%%%%%%%%%%%%%%%%%%%%%%%%%%%%%%%%%%%%%%%%%%%%%%%%%%%%%%%%%%%%%%%%%%%%%%%%%%%%%%%%%%%%%%%%%%%%%%%%%

\section{Transversal slices to conjugacy classes in algebraic
groups} \label{slices1}

\setcounter{equation}{0}

In this section we introduce algebraic group counterparts of the
Slodowy slices and prove an analogue of the Kostant cross--section
theorem for them.  The new slices will be associated to Weyl group elements and to certain parabolic subgroups defined with the help of the Weyl group elements.
We start with some preliminary facts about Weyl group actions on root systems.

Let $G$ be a complex semisimple (connected) algebraic group, $\g$ its Lie
algebra, $H$ a maximal torus in $G$. Denote by $\h$ the Cartan subalgebra in $\g$ corresponding to $H$. Let $\Delta$ be the root system of the pair $(\g,\h)$. For any root $\alpha\in \Delta$ we denote by $\alpha^\vee\in \h$ the corresponding coroot.

Let $s$ be an element of the Weyl group $W$ of the pair $(\g,\h)$ and $\h_{\mathbb{R}}$ the real form of $\h$, the real linear span of simple coroots in $\h$. The set of roots $\Delta$ is a subset of the dual space $\h_\mathbb{R}^*$.

The Weyl group element $s$ naturally acts on $\h_{\mathbb{R}}$ as an orthogonal transformation with respect to the scalar product induced by the Killing form of $\g$. Using the spectral theory of orthogonal transformations we can decompose $\h_{\mathbb{R}}$ into a direct orthogonal sum of $s$--invariant subspaces,
\begin{equation}\label{hdec}
\h_\mathbb{R}=\bigoplus_{i=0}^{K} \h_i,
\end{equation}
where we assume that $\h_0$ is the linear subspace of $\h_{\mathbb{R}}$ fixed by the action of $s$, and each of the other subspaces $\h_i\subset \h_\mathbb{R}$, $i=1,\ldots, K$, is either two--dimensional or one--dimensional and the Weyl group element $s$ acts on it as rotation with angle $\theta_i$, $0<\theta_i<\pi$ or as the reflection with respect to the origin, respectively. Note that since $s$ has finite order $\theta_i=\frac{2\pi}{m_i}$, $m_i\in \mathbb{N}$.

Since the number of roots in the root system $\Delta$ is finite one can always choose elements $h_i\in \h_i$, $i=0,\ldots, K$, such that $h_i(\alpha)\neq 0$ for any root $\alpha \in \Delta$ which is not orthogonal to the $s$--invariant subspace $\h_i$ with respect to the natural pairing between $\h_{\mathbb{R}}$ and $\h_{\mathbb{R}}^*$.

Now we consider certain $s$--invariant subsets of roots $\overline{\Delta}_i$, $i=0,\ldots, K$, defined as follows
\begin{equation}\label{di}
{\overline{\Delta}}_i=\{ \alpha\in \Delta: h_j(\alpha)=0, j>i,~h_i(\alpha)\neq 0 \},
\end{equation}
where we formally assume that $h_{K+1}=0$.
Note that for some indexes $i$ the subsets ${\overline{\Delta}}_i$ are empty, and that the definition of these subsets depends on the order of terms in direct sum (\ref{hdec}).

We also define other $s$--invariant subsets of roots ${\Delta}_{i_k}$, $k=0,\ldots, M$ for all indexes $i_k>0$, $k=1,\ldots, M$ such that $\overline{\Delta}_{i_k}$,
\begin{equation}\label{dik}
{\Delta}_{i_k}=\bigcup_{i_j\leq i_k}\overline{\Delta}_{i_j}.
\end{equation}
For convenience we assume that indexes $i_k$ are labeled in such a way that $i_j<i_k$ if and only if $j<k$.
According to this definition we have a chain of strict inclusions
\begin{equation}\label{inc}
\Delta_{i_M}\supset\Delta_{i_{M-1}}\supset\ldots\supset\Delta_{i_0},
\end{equation}
such that $\Delta_{i_M}=\Delta$, $\Delta_{0}=\{\alpha \in \Delta: s\alpha=\alpha\}$ is the set of roots fixed by the action of $s$, and ${\Delta}_{i_k}\setminus {\Delta}_{i_{k-1}}=\overline{\Delta}_{i_k}$. Observe also that the root system $\Delta$ is the disjoint union of the subsets $\overline{\Delta}_{i_k}$,
$$
\Delta=\bigcup_{k=0}^{M}\overline{\Delta}_{i_k}.
$$

Now assume that
\begin{equation}\label{cond}
|h_{i_k}(\alpha)|>|\sum_{l\leq j<k}h_{i_j}(\alpha)|, ~{\rm for~any}~\alpha\in \overline{\Delta}_{i_k},~k=0,\ldots, M,~l<k.
\end{equation}
Condition (\ref{cond}) can be always fulfilled by suitable rescalings of the elements $h_{i_k}$.

Consider the element
$$
\bar{h}=\sum_{k=0}^{M}h_{i_k}\in \h_\mathbb{R}.
$$
From definition (\ref{di}) of the sets $\overline{\Delta}_i$ we obtain that for $\alpha \in \overline{\Delta}_{i_k}$
\begin{equation}\label{dech}
\bar{h}(\alpha)=\sum_{j\leq k}h_{i_j}(\alpha)=h_{i_k}(\alpha)+\sum_{j< k}h_{i_j}(\alpha)
\end{equation}
Now condition (\ref{cond}), the previous identity and the inequality $|x+y|\geq ||x|-|y||$ imply that for $\alpha \in \overline{\Delta}_{i_k}$ we have
$$
|\bar{h}(\alpha)|\geq ||h_{i_k}(\alpha)|-|\sum_{j< k}h_{i_j}(\alpha)||>0.
$$
Since $\Delta$ is the disjoint union of the subsets $\overline{\Delta}_{i_k}$, $\Delta=\bigcup_{k=0}^{M}\overline{\Delta}_{i_k}$, the last inequality ensures that  $\bar{h}$ belongs to a Weyl chamber of the root system $\Delta$, and one can define the subset of positive roots $\Delta_+$ and the set of simple positive roots $\Gamma$ with respect to that chamber. From condition (\ref{cond}) and formula (\ref{dech}) we also obtain that a root $\alpha \in \overline{\Delta}_{i_k}$ is positive if and only if $h_{i_k}(\alpha)>0$.

To define the algebraic group analogues of the Slodowy slices we shall also need a parabolic subalgebra $\p$ of $\g$ associated to the semisimple element $\bar{h}_0=\sum_{k=0,i_k>0}^{M}h_{i_k}\in \h_\mathbb{R}$ associated to $s\in W$. This subalgebra is defined with the help of the linear eigenspace decomposition of $\g$ with respect to the adjoint action of $\bar{h}_0$ on $\g$, $\g=\bigoplus_{m}(\g)_m$, $(\g)_m=\{ x\in \g \mid [\bar{h}_0,x]=mx\}$, $m \in \mathbb{R}$. By definition $\p=\bigoplus_{m\leq 0}(\g)_m$ is a parabolic subalgebra in $\g$, $\n=\bigoplus_{m<0}(\g)_m$ and $\l=\{x\in \g \mid [\bar{h}_0,x]=0\}$ are the nilradical and the Levi factor of $\p$, respectively. We denote by $P$ the corresponding parabolic subgroup of $G$, by $N$ the unipotent radical of $P$ and by $L$ the Levi factor of $P$. The subgroups of $P$, $N$ and $L$ have Lie algebras $\p$, $\n$ and $\l$, respectively, and both $P$ and $L$ are connected. Note that we have natural inclusions of Lie algebras $\p\supset\b\supset\n$, where $\b$ is the Borel subalgebra of $\g$ corresponding to the system $-\Gamma$ of simple roots, and $\Delta_{0}$ is the root system of the reductive Lie algebra $\l$.

Let $X_\alpha\subset \g$ be the root subspace of $\g$ corresponding to root $\alpha \in \Delta$.
Fix a system of root vectors $e_\alpha\in X_{\alpha}, \alpha \in \Delta$ such that if $[e_{\alpha},e_\beta]=N_{\alpha,\beta}e_{\alpha+\beta}\in X_{\alpha+\beta}$ for any pair $\alpha,\beta\in \Gamma$ of simple positive roots then $[e_{-\alpha},e_{-\beta}]=N_{\alpha,\beta}e_{-\alpha-\beta}\in X_{-\alpha-\beta}$.

Recall that by Theorem 5.4.2. in \cite{GG1} one can uniquely choose a representative $s\in G$ for the Weyl group element $s\in W$ in such a way that the operator ${\rm Ad}s$ sends root vectors $e_{\pm \alpha}\in X_{\pm \alpha}$ to $e_{\pm s\alpha}\in X_{\pm s\alpha}$ for any simple positive root $\alpha \in \Gamma$.
We denote this representative by the same letter, $s\in G$. The representative $s\in G$ is called the normal representative of the Weyl group element $s\in W$. If the order of the Weyl group element $s\in W$ is equal to $R$ then the inner automorphism ${\rm Ad}s$ of the Lie algebra $\g$ has order at most $2R$, ${\rm Ad}s^{2R}={\rm id}$. We also recall that the operator ${\rm Ad}s$ sends each root subspace $X_\alpha\subset \g$, $\alpha \in \Delta$  to $X_{s\alpha}$.

The element $s\in G$ naturally acts on $G$
by conjugations. Let $Z$ be the set of $s$-fixed points in $L$,
\begin{equation}\label{defz}
Z=\{z\in L\mid szs^{-1}=z\},
\end{equation}
and
\begin{equation}\label{defns}
N_s=\{n\in N\mid sns^{-1}\in \overline N\},
\end{equation}
where $\overline N$ is the unipotent radical of the parabolic subgroup $\overline P\subset G$ opposite to $P$. Note that ${\rm dim}~N_s=l(s)$, where $l(s)$ is the length of the Weyl group element $s\in W$ with respect to the system $\Gamma$ of simple roots.
Clearly, $Z$ and $N_s$ are subgroups in $G$, and $Z$ normalizes
both $N$ and $N_s$. Denote by $\n_s$ and $\z$ the Lie algebras of
$N_s$ and $Z$, respectively.

Note that, since the operator ${\rm Ad}s$ sends root vectors $e_{\pm \alpha}\in X_{\pm \alpha}$ to $e_{\pm s\alpha}\in X_{\pm s\alpha}$ for any simple positive root $\alpha \in \Gamma$ and the root system of the reductive Lie algebra $\l$ is fixed by the action of $s$, the semisimple part $\m$ of the Levi subalgebra $\l$ is fixed by the action of ${\rm Ad}s$. In fact in this case $\z=\m\oplus \h_z$ and $\z \cap
\h= \mathbb{C}\h_0$, where $\h_z$ is a Lie subalgebra of the center of $\l$ and $\mathbb{C}\h_0$ is the linear subspace of $\h$ fixed by the action of $s$.

Now consider the subvariety $N_sZs^{-1}\subset G$. We shall prove that the variety $N_sZs^{-1}$ is transversal to the set of conjugacy classes in $G$. In order to do that we shall need
the following statement which is an
analogue of the Kostant cross--section theorem for the subvariety
$N_sZs^{-1}\subset G$.
\begin{proposition}\label{prop2}
Let $s\in W$ be an element of the Weyl group $W$ of the pair $(\g,\h)$.
Let $\overline{h}_0\in \h_\mathbb{R}$ be a semisimple element associated to $s$, $\overline{h}_0= \sum_{k=0,i_k>0}^{M}h_{i_k}$, where elements $h_{i_k}\in \h_{i_k}$ satisfy conditions (\ref{cond}) and $\h_{i_k}\subset \h_\mathbb{R}$ are the subspaces of $\h_\mathbb{R}$ defined in (\ref{hdec}). Let $\n\subset \g$ be the nilradical of the parabolic subalgebra $\p$ defined with the help of $\overline{h}_0$, $\n=\bigoplus_{m<0}(\g)_m$,$(\g)_m=\{ x\in \g \mid [\bar{h}_0,x]=mx\}$, $m \in \mathbb{R}$ and $\l=(\g)_0$ the Levi factor of $\p$. Denote by $P$, $N$ and $L$ the corresponding subgroups of $G$ and by $s\in G$ the normal representative of the Weyl group element $s\in W$. Let $Z$ be the centralizer of $s$ in $L$,
$$
Z=\{z\in L\mid szs^{-1}=z\},
$$
and
$$
N_s=\{n\in N\mid sns^{-1}\in \overline N\},
$$
where $\overline N$ is the unipotent radical of the parabolic subgroup $\overline P\subset G$ opposite to $P$.
Then the conjugation map
\begin{equation}\label{cross}
\alpha: N\times N_sZs^{-1}\rightarrow NZs^{-1}N
\end{equation}
is an isomorphism of varieties.
\end{proposition}

\begin{proof}
First observe that using a decomposition of $N$ as a product of one--dimensional subgroups corresponding to roots one can write $N=N_sN'_s$, where $N_s'=N\bigcap s^{-1}Ns$, and hence
$$
NZs^{-1}N=N_sN'_ss^{-1}NZ=N_ss^{-1}NZ=N_sZs^{-1}N.
$$

In order to prove that map (\ref{cross}) is an isomorphism is suffices to show that this map is bijective. Then by Zariski's main theorem the map $\alpha$ is an isomorphism of varieties.

Observe that map (\ref{cross}) is bijective if and only if for any given $n_s\in N_s, u\in N$ and $z\in Z$ the equation
\begin{equation}\label{tpr}
n_szs^{-1}u=nn_s'z's^{-1}n^{-1}
\end{equation}
has a unique solution $n\in N,n_s'\in N_s,z'\in Z$.
We prove the last statement by induction over certain $s$--invariant reductive subgroups in $G$ that we are going to define now. Consider the reductive Lie subalgebras $\g_{i_k}$, $k=0,\ldots ,M$ defined by induction as follows: $\g_{i_M}=\g$, $\g_{i_{k-1}}=\z_{\g_{i_k}}(h_{i_k})$, where $\z_{\g_{i_k}}(h_{i_k})$ is the centralizer of $h_{i_k}$ in $\g_{i_k}$. We denote by $G_{i_{k}}$ the corresponding subgroups in $G$.

By construction $\Delta_{i_k}$ is the root system of $\g_{i_k}$, and we have chains of strict inclusions
\begin{eqnarray}\label{incg}
\g=\g_{i_M}\supset\g_{i_{M-1}}\supset\ldots\supset\g_{0}=\l, \\
G=G_{i_M}\supset G_{i_{M-1}}\supset\ldots\supset G_{0}=L
\end{eqnarray}
corresponding to inclusions (\ref{inc}). Note that $\g_{i_{k-1}}$ is the Levi factor of the parabolic subalgebra $\p_{i_{k-1}}\subset \g_{i_k}$ associated to the semisimple element $h_{i_k}$ in the same way as $\p\subset \g$ is associated to $\bar{h}_0$, i.e. if $\g_{i_k}=\bigoplus_{m}(\g_{i_k})_m$, $(\g_{i_k})_m=\{ x\in \g_{i_k} \mid [h_{i_k},x]=mx\}$, $m \in \mathbb{R}$ then $\p_{i_{k-1}}=\bigoplus_{m\leq 0}(\g_{i_k})_m$, and ${\n}_{i_{k-1}}=\bigoplus_{m<0}(\g_{i_k})_m$ is the nilradical of $\p_{i_{k-1}}$, ${\g}_{i_{k-1}}=(\g_{i_k})_0$ is the Levi factor of $\p_{i_{k-1}}$. We also denote by $\overline{\n}_{i_{k-1}}=\bigoplus_{m>0}(\g_{i_k})_m$ the nilradical of the opposite parabolic subalgebra. Let $P_{i_{k-1}}$, ${N}_{i_{k-1}}$, and $\overline{N}_{i_{k-1}}$ be the corresponding subgroups of $G_{i_{k}}$. Below we shall need the following direct decompositions of linear spaces
\begin{eqnarray}\label{d1}
\g_{i_k}=\p_{i_{k-1}}+\overline{\n}_{i_{k-1}}={\n}_{i_{k-1}}+\g_{i_{k-1}}+\overline{\n}_{i_{k-1}}, \\
\n=\sum_{k=0}^{M-1}{\n}_{i_{k}} \label{d2}
\end{eqnarray}
following straightforwardly from the definitions of the subalgebras $\p_{i_{k-1}}$, ${\n}_{i_{k}}$ and $\overline{\n}_{i_{k-1}}$, and the root system decomposition
$\Delta_+=\bigcup_{k=0}^{M}(\overline{\Delta}_{i_k})_+$, $(\overline{\Delta}_{i_k})_+=\overline{\Delta}_{i_k}\bigcap \Delta_+=\{ \alpha \in {\Delta}_{i_k},h_{i_k}(\alpha)>0\}$. Decompositions (\ref{d1}) imply  decompositions of a dense subset $\overline{G}_{i_{k}}\subset G_{i_{k}}$,
\begin{equation}\label{d11}
\overline{G}_{i_{k}}=P_{i_{k-1}}\overline{N}_{i_{k-1}}={N}_{i_{k-1}}G_{i_{k-1}}\overline{N}_{i_{k-1}},
\end{equation}
and decomposition (\ref{d2}) implies two decompositions
\begin{equation}\label{d22}
N=N_{i_{M-1}}N_{i_{M-2}}\ldots N_{i_{0}}, N=N_{i_{0}}N_{i_{1}}\ldots N_{i_{M-1}}.
\end{equation}
Note that, since the subsets of roots $\Delta_{i_k}$ are $s$--invariant, the subalgebras $\g_{i_k}$ are ${\rm Ad}s$--invariant and the subgroups $G_{i_{k}}$ are invariant with respect to the action of $s$ on $G$ by conjugations.

Applying decomposition (\ref{d11}) successively we also obtain decompositions of dense subsets $\overline{G}_k\subset G$,
\begin{equation}\label{d111}
\overline{G}_{k}=N_{k}G_{i_{k}}\overline{N}_{k},~N_{k}=N_{i_{M-1}}N_{i_{M-2}}\ldots N_{i_{k}}, \overline{N}_{k}=\overline{N}_{i_{M-1}}\overline{N}_{i_{M-2}}\ldots \overline{N}_{i_{k}}.
\end{equation}

Due to the inclusions
\begin{equation}\label{ngcr}
[\g_{i_{k-1}},{\n}_{i_{k-1}}]\subset {\n}_{i_{k-1}}, [\g_{i_{k-1}},\overline{{\n}}_{i_{k-1}}]\subset \overline{{\n}}_{i_{k-1}}, {\n}_{i_{k-2}}\subset {\g}_{i_{k-1}}, \overline{{\n}}_{i_{k-2}}\subset {\g}_{i_{k-1}}
\end{equation}
$N_{k}$ and $\overline{N}_{k}$ are Lie subgroups in $G$ with Lie algebras $\sum_{m=k}^{M-1}{\n}_{i_{m}}$ and  $\sum_{m=k}^{M-1}\overline{{\n}}_{i_{m}}$, respectively.

We shall prove that equation (\ref{tpr}) has a unique solution by induction over the reductive subgroups $G_{i_k}$ starting with $k=0$. First we rewrite equation (\ref{tpr}) in a slightly different form,
\begin{equation}\label{tpr1}
n_szs^{-1}us=nn_s'z's^{-1}n^{-1}s,~n,u\in N,n_s,n_s'\in N_s,z,z'\in Z.
\end{equation}

To establish the base of induction we first observe that both the l.h.s. and the r.h.s. of equation (\ref{tpr1}) belong to the dense subset $\overline{G}_0\subset G$ and that
the $G_{0}=L$--component of equation (\ref{tpr1}) with respect to decomposition (\ref{d111}) for $k=0$ is reduced to \begin{equation}\label{eqz}
z=z'.
\end{equation}

Indeed, using a decomposition of $N$ as a product of one--dimensional subgroups corresponding to roots one can write $N=N'_{s^{-1}}N_{s^{-1}}$ and $N=N_{s^{-1}}N'_{s^{-1}}$, where $N_{s^{-1}}'=N\bigcap sNs^{-1}$, $N_{s^{-1}}=\{n\in N,s^{-1}ns\in \overline N\}$, and hence
\begin{equation}\label{sN}
s^{-1}Ns=s^{-1}N_{s^{-1}}'ss^{-1}N_{s^{-1}}s\subset N\overline{N}.
\end{equation}
If $u=u'u_{s^{-1}}$ and $n=n_{s^{-1}}n'$ are the decompositions of $u$ and $n$ corresponding to the decompositions $N=N'_{s^{-1}}N_{s^{-1}}$ and $N=N_{s^{-1}}N'_{s^{-1}}$, respectively, then recalling that $Z$ normalizes both $N$ and $\overline{N}$ we deduce that the decompositions of the r.h.s. and of the l.h.s. of equation (\ref{tpr1}) corresponding to the decomposition $\overline{G}_0=NL\overline{N}$ take the form
$$
(n_szs^{-1}u'sz^{-1})z(s^{-1}u_{s^{-1}}s)=(nn_s'z's^{-1}{n'}^{-1}s{z'}^{-1})z'(s^{-1}n_{s^{-1}}^{-1}s),
$$
where $n_szs^{-1}u'sz^{-1},nn_s'z's^{-1}{n'}^{-1}s{z'}^{-1}\in N$ and $s^{-1}u_{s^{-1}}s, s^{-1}n_{s^{-1}}^{-1}s\in \overline{N}$, $z,z'\in Z\subset L$. This implies (\ref{eqz}) and establishes the base of induction.

Now let
\begin{equation}\label{decn}
n=n_{i_{M-1}}n_{i_{M-2}}\ldots n_{i_0},u=u_{i_{0}}u_{i_{2}}\ldots u_{i_{M-1}},~n_s={n_s}_{i_{M-1}}{n_s}_{i_{M-2}}\ldots {n_s}_{i_0},~n_s'={n_s'}_{i_0}{n_s'}_{i_2}\ldots {n_s'}_{i_{M-1}}
\end{equation}
be the decompositions of the elements $n,u,n_s$ and $n_s'$ corresponding to decompositions (\ref{d22}) and assume that $n_{i_j}$ and ${n_s'}_{i_j}$ have already been uniquely defined for $j<k-1$. We shall show that using equation (\ref{tpr1}) one can find $n_{i_{k-1}}$ and ${n_s'}_{i_{k-1}}$ in a unique way.

Observe that both the l.h.s. and the r.h.s. of equation (\ref{tpr1}) belong to the dense subset $\overline{G}_k\subset G$ and that
the $G_{i_k}$--component of equation (\ref{tpr1}) with respect to decomposition (\ref{d111}) is reduced to
\begin{equation}\label{tpr2}
{n_s}_{i_{k-1}}({n}_s)_{k-1}zs^{-1}({u})_{k-1}u_{i_{k-1}}s=
n_{i_{k-1}}({n})_{k-1}(n'_s)_{k-1}{n_s'}_{i_{k-1}}zs^{-1}({n})_{k-1}^{-1}n_{i_{k-1}}^{-1}s,
\end{equation}
where $({n}_s)_{k-1}={n_s}_{i_{k-2}}\ldots {n_s}_{i_0}\in G_{i_{k-1}}$, $({n})_{k-1}={n}_{i_{k-2}}\ldots {n}_{i_0}\in G_{i_{k-1}}$, $(n'_s)_{k-1}={n_s'}_{i_0}\ldots {n_s'}_{i_{k-2}}\in G_{i_{k-1}}$, $({u})_{k-1}={u}_{i_{0}}\ldots {u}_{i_{k-2}}\in G_{i_{k-1}}$ and $({n})_{k-1}, ({n}_s)_{k-1}, (n'_s)_{k-1}, ({u})_{k-1}, {n_s}_{i_{k-1}}, u_{i_{k-1}}, z$ are already known.
This follows, similarly to the case $k=0$, from decompositions (\ref{decn}), inclusions (\ref{ngcr}), which also imply that $G_{i_{k-1}}$ normalizes both $N_{i_{k-1}}$ and $\overline{N}_{i_{k-1}}$, and the fact that the subgroups $G_{i_{k}}$ are invariant with respect to the action of $s$ on $G$ by conjugations.

Now recall that $n_{i_{k-1}}\in N_{i_{k-1}}$ and that the subgroup $N_{i_{k-1}}$ is generated by one--dimensional subgroups corresponding to roots $\alpha \in -(\overline{\Delta}_{i_k})_+$.
 Recall also that a root $\alpha\in \overline{\Delta}_{i_k}$ belongs to the set $-(\overline{\Delta}_{i_k})_+$ if and only if $h_{i_k}(\alpha)<0$. Identifying $\h_\mathbb{R}$ and $\h_\mathbb{R}^*$ with the help of the Killing form one can orthogonally project the root $\alpha$ onto the two--dimensional plane (or the line) $\h_{i_k}$ containing $h_{i_k}$. We denote this projection by $\alpha_{i_k}$. Now we conclude that $\alpha$ belongs to the set $-(\overline{\Delta}_{i_k})_+$ if and only if the angle between $\alpha_{i_k}$ and $h_{i_k}$ in the plane $\h_{i_k}$ is obtuse or equal to $\pi$.

Using this observation and noting that $s$ acts in the plane $\h_{i_k}$ by rotation with angle $\theta_{i_k}=\frac{2\pi}{m_{i_k}}\leq \pi,m_{i_k}\in \mathbb{N}$ we deduce that the set $-(\overline{\Delta}_{i_k})_+$ is the disjoint union of the subsets $\overline{\Delta}_{i_k}^l=\{\alpha \in -(\overline{\Delta}_{i_k})_+: s^{-1}\alpha, \ldots, s^{-l}\alpha \in -(\overline{\Delta}_{i_k})_+, s^{-(l+1)}\alpha\in (\overline{\Delta}_{i_k})_+\}$, $l=0, \ldots, D_k+1$.
This can be easily seen from Figure 1 at which the plane $\h_{i_k}$ is shown.
$$
\xy/r10pc/: ="A",-(1,0)="B","A",-(0,0.57)*{h_{i_k}},"A", +(1,0.23)*{\overline{\Delta}_{i_k}^0}, +(-0.22,0.47)*{\overline{\Delta}_{i_k}^1}, +(-0.39,0.25)*{\overline{\Delta}_{i_k}^2}, "B", +(0.1,0.1)*{\overline{\Delta}_{i_k}^{D_k+1}},+(0.06,0.33)*{\overline{\Delta}_{i_k}^{D_k}},"A",
{\xypolygon13"C"{~<{-}~>{}{}}},  "C0";"B"**@{--}, "C0", {\ar+(0,-0.5)},"A", +(0.3,0)="D",+(-0.038,0.14)="E","D";"E" **\crv{(1.3,0.07)}, "A", +(0.4,0.08)*{\theta_{i_k}}
\endxy
$$
\begin{center}
 Fig.1
\end{center}
The vector $h_{i_k}$ is directed downwards at the picture, and the orthogonal projections of elements from $-(\overline{\Delta}_{i_k})_+$ onto $\h_{i_k}$ are contained in the upper half plane. The projections of the roots from the subset $\overline{\Delta}_{i_k}^l$ onto $\h_{i_k}$ are contained in the interior of the sector labeled by $\overline{\Delta}_{i_k}^l$ at Figure 1. All those sectors belong to the upper half plane and have the same central angles equal to ${\theta_{i_k}}$, except for the last sector labeled by $\overline{\Delta}_{i_k}^{D_k+1}$, which can possibly have a smaller angle. The element $s\in W$ acts on the plane $\h_{i_k}$ by counterclockwise rotation with angle ${\theta_{i_k}}$.

Now consider the unipotent subgroups $N_{i_{k-1}}^l$, $l=0, \ldots, D_k+1$, generated by the one--dimensional subgroups corresponding to the roots from the sets $\overline{\Delta}_{i_k}^l$.

Obviously we have decompositions
\begin{equation}\label{decnk}
N_{i_{k-1}}=N_{i_{k-1}}^{0}N_{i_{k-1}}^{1}\ldots N_{i_{k-1}}^{{D_k+1}},~ N_{i_{k-1}}=N_{i_{k-1}}^{{D_k+1}}N_{i_{k-1}}^{{D_k}}\ldots N_{i_{k-1}}^{0}.
\end{equation}
Let
\begin{equation}\label{decnu}
n_{i_{k-1}}=n_{i_{k-1}}^{0}n_{i_{k-1}}^{1}\ldots n_{i_{k-1}}^{{D_k+1}},~ u_{i_{k-1}}=u_{i_{k-1}}^{{D_k+1}}u_{i_{k-1}}^{{D_k}}\ldots u_{i_{k-1}}^{0}
\end{equation}
be the corresponding decomposition of elements $n_{i_{k-1}}$ and $u_{i_{k-1}}$, respectively.

By definition the subgroups $N_{i_{k-1}}^l$ have the following properties:
\begin{eqnarray}\label{sap}
s^{-p}N_{i_{k-1}}^{p-1}s^p\subset \overline{N}_{i_{k-1}},~p=1,\ldots, D_k+2, \\
s^{-1}N_{i_{k-1}}^{p}s\subset {N}_{i_{k-1}}^{p-1},~p=1,\ldots, D_k+1. \label{sap1}
\end{eqnarray}

From Figure 1 it also can be seen that the elements ${n_s}_{i_{k-1}}$ and ${n_s'}_{i_{k-1}}$ have the following decompositions
\begin{equation}\label{decns}
{n_s}_{i_{k-1}}={n_s}_{i_{k-1}}^{{D_k}}{n_s}_{i_{k-1}}^{{D_k+1}},
~{n_s'}_{i_{k-1}}={n_s'}_{i_{k-1}}^{{D_k+1}}{n_s'}_{i_{k-1}}^{{D_k}},
\end{equation}
where ${n_s}_{i_{k-1}}^{{D_k}},{n_s'}_{i_{k-1}}^{{D_k}}\in N_{i_{k-1}}^{{D_k}}$, and ${n_s}_{i_{k-1}}^{{D_k+1}},{n_s'}_{i_{k-1}}^{{D_k+1}}\in N_{i_{k-1}}^{{D_k+1}}$.

We claim that the components $n_{i_{k-1}}^{l}$, $l=0,\ldots, D_k+1$ can be uniquely calculated by induction starting with $n_{i_{k-1}}^{0}$, and at the last two steps of the induction the element ${n_s'}_{i_{k-1}}$ is also uniquely computed.

The induction procedure is based on a very simple observation that can be visualized with the help of Figure 1. Under the action of the operator $s^{-1}$ the orthogonal projections onto $\h_{i_k}$ of the roots from $\overline{\Delta}_{i_k}^l$ are moved into the sector labeled $\overline{\Delta}_{i_k}^{l-1}$ on Figure 1, and each sector labeled by $\overline{\Delta}_{i_k}^l$, $l=0,\ldots, D_k+1$ is moved to the lower half plane by applying a large enough power of the operator $s^{-1}$.

To establish the base of the induction we observe that using decompositions (\ref{decnu}) one can immediately deduce that both the l.h.s. and the r.h.s. of equation (\ref{tpr2}) belong to the dense subset $\overline{G}_{i_k}\subset G_{i_k}$, and the $\overline{N}_{i_{k-1}}$--component of
equation (\ref{tpr2}) with respect to decomposition (\ref{d11}) takes the form
$$
s^{-1}u_{i_{k-1}}^{0}s=s^{-1}(n_{i_{k-1}}^{0})^{-1}s,
$$
and hence $n_{i_{k-1}}^{0}=(u_{i_{k-1}}^{0})^{-1}$.

Now assume that elements $n_{i_{k-1}}^{l}$, $l=0, \ldots, p-1$, $p\leq D_k$ are uniquely defined. We show that the component $n_{i_{k-1}}^{p}$ can be found uniquely from equation (\ref{tpr2}). Indeed, consider the element $m_p=n_{i_{k-1}}^{0}n_{i_{k-1}}^{1}\ldots n_{i_{k-1}}^{{p-1}}$. Multiplying  equation (\ref{tpr2}) by the element $m_p^{-1}$ from the left and by $s^{-1}m_ps$ from the right and conjugating the result by $s^{-p}$ we obtain
\begin{eqnarray}\label{tpr3}
s^{-p}m_p^{-1}{n_s}_{i_{k-1}}({n}_s)_{k-1}zs^{-1}({u})_{k-1}u_{i_{k-1}}m_ps^{p+1}= \qquad \qquad \qquad \qquad \qquad \qquad \qquad \qquad \\
=s^{-p}n_{i_{k-1}}^p\ldots n_{i_{k-1}}^{D_k+1}({n})_{k-1}(n'_s)_{k-1}{n_s'}_{i_{k-1}}zs^ps^{-p-1}({n})_{k-1}^{-1}({n_{i_{k-1}}^{D_k+1}})^{-1}\ldots ({n_{i_{k-1}}^{p}})^{-1}s^{p+1}. \nonumber
\end{eqnarray}
By the induction hypothesis the l.h.s. of (\ref{tpr3}) is uniquely expressed in terms of $n_s,z$ and $u$. Similarly to the case $l=0$ one can show that both the l.h.s. and the r.h.s. of equation (\ref{tpr3}) belong to the dense subset $\overline{G}_{i_k}\subset G_{i_k}$, and by the induction hypothesis and by properties (\ref{sap}) the $\overline{N}_{i_{k-1}}$--component of
equation (\ref{tpr3}) with respect to decomposition (\ref{d11}) takes the form
$$
x_p=s^{-p-1}({n_{i_{k-1}}^{p}})^{-1}s^{p+1},
$$
where $x_p\in s^{-p-1}{N_{i_{k-1}}^{p}}s^{p+1}=s^{-1}{N_{i_{k-1}}^{0}}s$ is expressed in terms of $n_s,z$ and $u$. Therefore $n_{i_{k-1}}^{p}=s^{p+1}x_p^{-1}s^{-p-1}$.

If $p=D_k+1$ then we rewrite equation (\ref{tpr3}) in the following form
\begin{eqnarray}\label{tpr4}
s^{-D_k-1}m_p^{-1}{n_s}_{i_{k-1}}({n}_s)_{k-1}zs^{-1}({u})_{k-1}u_{i_{k-1}}m_ps^{D_k+1}= \qquad \qquad \qquad \qquad \qquad \qquad \qquad \qquad \\
=s^{-D_k-1} n_{i_{k-1}}^{D_k+1}({n})_{k-1}(n'_s)_{k-1}{n_s'}_{i_{k-1}}^{D_k+1}s^{D_k+1}z(s^{-D_k-2}({n})_{k-1}^{-1}s^{D_k+2})v_ks^{-D_k-2}({n_{i_{k-1}}^{D_k+1}})^{-1}s^{D_k+2}, \nonumber
\end{eqnarray}
where $v_k=s^{-D_k-1}s^{-1}({n})_{k-1}sz^{-1}{n_s'}_{i_{k-1}}^{D_k}z
s^{-1}({n})_{k-1}^{-1}ss^{D_k+1}$.
From the definition of the subgroups $N_{i_{k-1}}^{p}$ and condition (\ref{cond}) it follows that the subgroup $G_{i_{k-1}}$ normalizes each of the subgroups $N_{i_{k-1}}^{p}$. Therefore recalling that $Z$ also normalizes each of the subgroups $N_{i_{k-1}}^{p}$ we deduce $s^{-1}({n})_{k-1}sz^{-1}{n_s'}_{i_{k-1}}^{D_k}z
s^{-1}({n})_{k-1}^{-1}s\in {N}_{i_{k-1}}^{D_k}$.
Now similarly to the case $l=0$ one can show that both the l.h.s. and the r.h.s. of equation (\ref{tpr4}) belong to the dense subset $\overline{G}_{i_k}\subset G_{i_k}$, and by the induction hypothesis and by properties (\ref{sap}) the $\overline{N}_{i_{k-1}}$--component of
equation (\ref{tpr3}) with respect to decomposition (\ref{d11}) takes the form
\begin{equation}\label{ll}
x_{D_k+1}=v_ks^{-D_k-2}({n_{i_{k-1}}^{D_k+1}})^{-1}s^{D_k+2},
\end{equation}
where $x_{D_k+1}\in s^{-1}{N_{i_{k-1}}^{0}}s$ is expressed in terms of $n_s,z$ and $u$.

By definition the element ${n_s'}_{i_{k-1}}^{D_k}$ belongs to the unipotent subgroup $N_{s}^{{D_k}}=N_s\bigcap N_{i_{k-1}}^{{D_k}}$ generated by one--dimensional subgroups corresponding to roots from the set $\Delta_s^{{D_k}}=s^{-1}(\Delta_+)\bigcap \overline{\Delta}_{i_k}^{D_k}$. This subgroup is normalized by $G_{i_{k-1}}$ since both $N_s$ and $N_{i_{k-1}}^{{D_k}}$ are normalized by $G_{i_{k-1}}$ by condition (\ref{cond}). Therefore $s^{-1}({n})_{k-1}sz^{-1}{n_s'}_{i_{k-1}}^{D_k}z
s^{-1}({n})_{k-1}^{-1}s\in N_{s}^{{D_k}}$, and $v_k\in s^{-D_k-1}N_{s}^{{D_k}}s^{D_k+1}$.
Recall also that the element ${n}_{i_{k-1}}^{D_k+1}$ belongs to the unipotent subgroup $N_{i_{k-1}}^{{D_k+1}}$ generated by one--dimensional subgroups corresponding to roots from the set $\overline{\Delta}_{i_k}^{D_k+1}=s^{-1}(\Delta_+)\bigcap \overline{\Delta}_{i_k}^{D_k+1}$.

Now consider the closed subset of roots $s^{-1}(\overline{\Delta}_{i_k}^{0})=\Delta_+\bigcap s^{-1}(\overline{\Delta}_{i_k})\subset \Delta_+$ and the corresponding unipotent subgroup $s^{-1}N_{i_{k-1}}^{0}s\subset \overline{N}_{i_{k-1}}$.
Observe that the closed subset of roots $s^{-1}(\overline{\Delta}_{i_k}^{0})$ can be represented as the following disjoint union $s^{-1}(\overline{\Delta}_{i_k}^{0})=s^{-D_k-1}(\Delta_s^{{D_k}})\bigcup s^{-D_k-2}(\overline{\Delta}_{i_k}^{D_k+1})$, and hence we have a unique factorization
\begin{equation}\label{ff}
s^{-1}N_{i_{k-1}}^{0}s=s^{-D_k-1}N_{s}^{{D_k}}s^{D_k+1}s^{-D_k-2}N_{i_{k-1}}^{{D_k+1}}s^{D_k+2}
\end{equation}
Since in formula (\ref{ll}) $v_k\in s^{-D_k-1}N_{s}^{{D_k}}s^{D_k+1}$ and $s^{-D_k-2}({n_{i_{k-1}}^{D_k+1}})^{-1}s^{D_k+2}\in s^{-D_k-2}N_{i_{k-1}}^{{D_k+1}}s^{D_k+2}$ both $v_k$ and $s^{-D_k-2}({n_{i_{k-1}}^{D_k+1}})^{-1}s^{D_k+2}$ can be uniquely defined in terms of the corresponding components of $x_{D_k+1}$, $x_{D_k+1}=x_{D_k+1}'x_{D_k+1}''$,
$x_{D_k+1}'\in s^{-D_k-1}N_{s}^{{D_k}}s^{D_k+1},x_{D_k+1}''\in s^{-D_k-2}N_{i_{k-1}}^{{D_k+1}}s^{D_k+2}$,
$$
v_k=x_{D_k+1}',~s^{-D_k-2}({n_{i_{k-1}}^{D_k+1}})^{-1}s^{D_k+2}=x_{D_k+1}'',
$$
and hence
$$
{n_s'}_{i_{k-1}}^{D_k}=zs^{-1}({n})_{k-1}^{-1}ss^{D_k+1}x_{D_k+1}'s^{-D_k-1}s^{-1}({n})_{k-1}sz^{-1},~~
{n_{i_{k-1}}^{D_k+1}}=s^{D_k+2}x_{D_k+1}''^{-1}s^{-D_k-2}.
$$

Finally, one can uniquely compute the only remaining unknown ${n_s'}_{i_{k-1}}^{D_k+1}$ from equation (\ref{tpr2}) in a similar way.
This proves the induction step and concludes the proof of the theorem.

\end{proof}

\begin{remark}
The particular choice of the normal representative $s\in G$ for a Weyl group element $s\in W$ in the previous proposition is technically convenient but actually not important. Actually any representative of $s\in W$ in the normalizer of $\h$ in $G$ is $H$--conjugate to an element from $Zs$, where $s\in G$ is the normal representative of $s\in W$.
\end{remark}

Now we can prove the main statement of this section.

\begin{proposition}\label{prop1}
Under the conditions of Proposition \ref{prop2} the variety $N_sZs^{-1}\subset G$ is a transversal
slice to the set of conjugacy classes in $G$.
\end{proposition}
\begin{proof} We have to show that the conjugation map
\begin{equation}\label{map}
\gamma:G\times N_sZs^{-1}\rightarrow G
\end{equation}
has the surjective differential.

Note that the set of smooth points of map (\ref{map}) is stable
under the $G$--action by left translations on the first factor of
$G\times N_sZs^{-1}$. Therefore it suffices to show that the
differential of map (\ref{map}) is surjective at points
$(1,n_szs^{-1})$, $n_s\in N_s$, $z\in Z$.

In terms of the right trivialization of the tangent bundle $TG$
and the induced trivialization of $T(N_sZs^{-1})$ the differential
of map (\ref{map}) at points $(1,n_szs^{-1})$ takes the form
\begin{eqnarray}\label{diff8*}
d\gamma_{(1,n_szs^{-1})}:(x,(n,w))\rightarrow (Id-{\rm
Ad}(n_szs^{-1}))x+n+w, \\x\in\g\simeq T_1(G), (n,w)\in
\n_s+\z\simeq T_{n_szs^{-1}}(N_sZs^{-1}). \nonumber
\end{eqnarray}

In order to show that the image of map (\ref{diff8*}) coincides
with $T_{n_szs^{-1}}G\simeq \g$ we shall need a direct orthogonal, with respect to the Killing form, vector space decomposition of the Lie algebra $\g$,
\begin{equation}\label{dec}
\g=\n+\z+\opn+\h_0^\perp,
\end{equation}
where $\opn$ is the nilradical of the parabolic subalgebra $\overline{\p}$ opposite to $\p$, and $\h_0^\perp$ is the orthogonal complement in $\h$ to $\mathbb{C}\h_0$.

We shall use the map $\alpha: N\times N_sZs^{-1}\rightarrow N_sZs^{-1}N$ introduced in Proposition \ref{prop2}. By definition $\alpha$ is the restriction of the map $\gamma$ to the subset $N\times N_sZs^{-1}\subset G\times N_sZs^{-1}$. Observe that in terms of the right trivialization of the tangent bundle $TG$ the differential of the map $\alpha$ at points $(1,n_szs^{-1})\in N\times N_sZs^{-1}$, $n_s\in N_s,z\in Z$ is given by
\begin{eqnarray}\label{diff8}
d\alpha_{(1,n_szs^{-1})}:(x,(n,w))\rightarrow (Id-{\rm
Ad}(n_szs^{-1}))x+n+w, \\x\in\n\simeq T_1(N), (n,w)\in
\n_s+\z\simeq T_{n_szs^{-1}}(N_sZs^{-1}). \nonumber
\end{eqnarray}
Recall that by Proposition \ref{prop2} the conjugation map $\alpha: N\times N_sZs^{-1}\rightarrow N_sZs^{-1}N$ is an isomorphism, and hence its differential is an isomorphism of the corresponding tangent spaces at all points. Using the right trivialization of the tangent bundle $TG$ the tangent space $T_{n_szs^{-1}}(N_sZs^{-1}N)$ can be identified with $\n_s+\z+{\rm Ad}(n_szs^{-1})\n$, $T_{n_szs^{-1}}(N_sZs^{-1}N)=\n_s+\z+{\rm Ad}(n_szs^{-1})\n$. Therefore (\ref{diff8}) and Proposition \ref{prop2} imply that
\begin{equation}\label{inc*}
(Id-{\rm
Ad}(n_szs^{-1}))\n+\n_s+\z={\rm Ad}(n_szs^{-1})\n +\n_s+\z.
\end{equation}
Now observe that by definition the subset $(Id-{\rm Ad}(n_szs^{-1}))\n\subset \g$ is contained in the image of $d\alpha_{(1,n_szs^{-1})}$, and by (\ref{inc*}) the subset ${\rm Ad}(n_szs^{-1})\n\subset \g$ is also contained in the image of $d\alpha_{(1,n_szs^{-1})}$.
Since $\n=(Id-{\rm Ad}(n_szs^{-1}))\n+{\rm Ad}(n_szs^{-1})\n$, we deduce that $\n$ is contained in the image of $d\alpha_{(1,n_szs^{-1})}$, and hence in the image of $d\gamma_{(1,n_szs^{-1})}$,
\begin{equation}\label{inc1}
\n \subset {\rm Im}~d\gamma_{(1,n_szs^{-1})}.
\end{equation}

Interchanging the roles $N$ and $\overline{N}$ in Proposition \ref{prop2} we immediately obtain that the conjugation map $\overline{\alpha}: \overline{N}\times s^{-1}\overline{N}_sZ\rightarrow \overline{N}s^{-1}\overline{N}_sZ$, $\overline{N}_s=sN_ss^{-1}$, is an isomorphism. Observe also that by definition the map $\overline{\alpha}$ is the restriction of $\gamma$ to the subset $\overline{N}\times s^{-1}\overline{N}_sZ=\overline{N}\times N_sZs^{-1}\subset G \times N_sZs^{-1}$. Now using the differential of the map $\overline{\alpha}$ we immediately infer, similarly to inclusion (\ref{inc1}), that
\begin{equation}\label{inc2}
\overline{\n} \subset {\rm Im}~d\gamma_{(1,n_szs^{-1})}.
\end{equation}

Now observe that since $\h$ normalizes $\n$, and $Z$ is a subgroup of $L$, which is, by definition, the centralizer of $\h_0^\perp$, we have for any $x_{\h_0^\perp}\in {\h_0^\perp}$
\begin{equation}\label{s22}
\left((Id-{\rm
Ad}(n_szs^{-1}))x_{\h_0^\perp}\right)_{\h_0^\perp}=(Id-{\rm
Ad}s^{-1})x_{\h_0^\perp}.
\end{equation}
Since by definition the operator ${\rm
Ad}s^{-1}$ has no fixed points in the invariant subspace ${\h_0^\perp}$, from formula (\ref{s22}) it follows that ${\h_0^\perp}$ is contained in the image of $(Id-{\rm Ad}(n_szs^{-1}))$, and hence, by formula (\ref{diff8*}), in the image of $d\gamma_{(1,n_szs^{-1})}$. Recalling also inclusions (\ref{inc1}) and (\ref{inc2}) and taking into account the obvious inclusion $\z\subset {\rm Im}~d\gamma_{(1,n_szs^{-1})}$ and decomposition (\ref{dec}) we deduce that the image of the map $d\gamma_{(1,n_szs^{-1})}$ coincides with $\g \simeq T_{n_szs^{-1}}G$. Therefore the differential of the map $\gamma$ is surjective at all points.

\end{proof}

%%%%%%%%%%%%%%%%%%%%%%%%%%%%%%%%%%%%%%%%%%%%%%%%%%%%%%%%%%%%%%%%%%%%%%%%%%%%%%%%%%%%%%%%%%%%%%%%%%%%%%%%%%%%%%%%%%%

\section{Subregular slices and simple singularities}\label{sectsub}

\setcounter{equation}{0}

Recall that the definition of transversal slices to adjoint orbits in a complex simple Lie algebra $\g$ and to conjugacy classes in a complex simple algebraic group
$G$ given in \cite{SL} was motivated by the study of simple singularities. Simple singularities appear in algebraic group theory as some singularities of the fibers
of the conjugation quotient map $\delta_G: G \rightarrow H/W$ generated by the inclusion $\mathbb{C}[H]^W\simeq
\mathbb{C}[G]^G\hookrightarrow \mathbb{C}[G]$, where $H$ is a maximal torus of $G$ and $W$ is the Weyl group of the pair $(G, H)$. Some fibers of this map are singular,
 the singularities correspond to irregular elements of $G$, and one can study these singularities by restricting $\delta_G$ to certain transversal slices to
 conjugacy classes in $G$. Simple singularities can be identified with the help of the following proposition proved in \cite{SL}.

\begin{proposition}{\bf (\cite{SL}, Section 6.5)}\label{sld}
Let $S$ be a transversal slice for the conjugation action of $G$ on itself. Assume that $S$  has dimension $r+2$, where $r$ is the rank of $G$.
Then the fibers of the restriction of the adjoint quotient map to $S$, $\delta_G:S \rightarrow H/W$, are normal surfaces with isolated singularities.
A point $x\in S$ is an isolated singularity of such a fiber iff $x$ is subregular in $G$, and $S$ can be regarded as a deformation of this singularity.

Moreover, if $t\in H/W$, and $x\in S$ is a singular point of the fiber $\delta_G^{-1}(t)$, then $x$ is a rational double point of type $_h\Delta_i$ for a suitable
$i\in\{1,\ldots,m\}$, where $\Delta_i$ are the components in the decomposition of the Dynkin diagram $\Delta(t)$ of the centralizer $Z_G(t)$ of $t$ in $G$,
$\Delta(t)=\Delta_1\cup\ldots \cup \Delta_m$. If $\Delta_i$ is of type $A$, $D$ or $E$ then $_h\Delta_i=\Delta_i$; otherwise $_h\Delta_i$ is the homogeneous diagram
of type $A$, $D$ or $E$ associated to $\Delta_i$ by the rule $_hB_n=A_{2n-1}$, $_hC_n=D_{n+1}$, $_hF_4=E_6$, $_hG_2=D_4$.
\end{proposition}

Note that all simple singularities of types $A$, $D$ and $E$ were explicitly constructed in \cite{SL} using certain special transversal slices in complex simple Lie algebras $\g$ and the adjoint quotient map $\delta_{\g}: \g \rightarrow \h/W$ generated by the inclusion $\mathbb{C}[\h]^W\simeq
\mathbb{C}[\g]^G\hookrightarrow \mathbb{C}[\g]$, where $\h$ is a Cartan subalgebra of $\g$ and $W$ is the Weyl group of the pair $(\g, \h)$. Below we give an alternative description of simple singularities in terms of transversal slices in algebraic groups. In our construction we shall use the transversal slices defined in Section \ref{slices1}. The corresponding elements $s\in W$ will be associated to subregular nilpotent elements $e$ in $\g$ via the Kazhdan--Lusztig map introduced in \cite{KL}(recall that $e\in \g$ is subregular if the dimension of its centralizer in $\g$ is equal to ${\rm rank}~\g+2$). The Kazhdan--Lusztig map is a certain mapping from the set of nilpotent adjoint orbits in $\g$ to the set of conjugacy classes in $W$. In this paper we do not need the definition of this map in the general case.
We shall only  describe  the values of this map on the subregular nilpotent adjoint orbits.

\begin{lemma}\label{subreg}
Let $\g$ be a complex simple Lie algebra, $\b$ a Borel subalgebra of $\g$ containing a Cartan subalgebra $\h\subset \b$. Let $W$ be the Weyl group of the pair $(\g, \h)$. Denote by $\Gamma=\{\alpha_1,\ldots,\alpha_r\}$, $r={\rm rank}~\g$ the corresponding system of simple positive roots of $\g$ and by $\Delta$ the root system of $\g$. Fix a system of root vectors $e_{\alpha}\in \g$, $\alpha \in \Delta$. One can choose a representative $e$ in the unique subregular nilpotent adjoint orbit of $\g$ and a representative $s_e$ in the conjugacy class in $W$, which corresponds to $e$ under the Kazhdan--Lusztig map, as follows (below we use the convention of \cite{E} for the numbering of simple roots; for the exceptional lie algebras we give the type of $e$ according to classification \cite{BK} and the type of $s_e$ according to classification \cite{C}):

\begin{itemize}

\item
$A_{r}$, $\g=\mathfrak{sl}_{r+1}$,
$$e=e_{\alpha_1}+\ldots +e_{\alpha_{r-1}},$$
the class of $s_e$ is the Coxeter class in a root subsystem $A_{r-1}\subset A_{r}$, and all such subsystems are $W$--conjugate,
$$
s_e=s_{\alpha_1}\ldots s_{\alpha_{r-1}};$$

\item
$B_r$, $\g=\mathfrak{so}_{2r+1}$,
$$e=e_{\alpha_1}+\ldots +e_{\alpha_{r-2}}+e_{\alpha_{r-1}+\alpha_{r}}+e_{\alpha_{r}},$$
the class of $s_e$ is the Coxeter class in a root subsystem $D_{r}\subset B_{r}$, and all such subsystems are $W$--conjugate,
$$
s_e=s_{\alpha_1}\ldots s_{\alpha_{r-2}}s_{\alpha_{r-1}}s_{\alpha_{r-1}+2\alpha_{r}};$$

\item
$C_r$, $\g=\mathfrak{sp}_{2r}$,
$$e=e_{\alpha_1}+\ldots +e_{\alpha_{r-2}}+e_{2\alpha_{r-1}+\alpha_{r}}+e_{\alpha_{r}},$$
the class of $s_e$ is the Coxeter class in a root subsystem $C_{r-1}+A_1\subset C_{r}$, and all such subsystems are $W$--conjugate,
$$
s_e=s_{\alpha_1}\ldots s_{\alpha_{r-2}}s_{2\alpha_{r-1}+\alpha_{r}}s_{\alpha_{r}};$$

\item
$D_r$, $\g=\mathfrak{so}_{2r}$,
$$e=e_{\alpha_1}+\ldots +e_{\alpha_{r-4}}+e_{\alpha_{r-3}+\alpha_{r-2}}+e_{\alpha_{r-2}+\alpha_{r-1}}+e_{\alpha_{r-1}}+e_{\alpha_{r}},$$
the class of $s_e$ is the class $D(a_1)$,
$$
s_e=s_{\alpha_1}\ldots s_{\alpha_{r-2}}s_{\alpha_{r-1}}s_{\alpha_{r-2}+\alpha_{r-1}+\alpha_{r}};$$

\item
$E_6$, the type of $e$ is $E_6(a_1)$,
$$e=e_{\alpha_1} +e_{\alpha_{2}+\alpha_{3}}+e_{\alpha_{4}}+e_{\alpha_{5}}+e_{\alpha_{3}+\alpha_{6}}+e_{\alpha_{6}},$$
$s_e$ has type $E_6(a_1)$,
$$
s_e=s_{\alpha_1}s_{\alpha_{2}+\alpha_{3}}s_{\alpha_{4}}s_{\alpha_{5}}s_{\alpha_{3}+\alpha_{6}}s_{\alpha_{6}};$$

\item
$E_7$, the type of $e$ is $E_7(a_1)$,
$$e=e_{\alpha_1} + e_{\alpha_2} +e_{\alpha_{3}+\alpha_{4}}+e_{\alpha_{5}}+e_{\alpha_{6}}+e_{\alpha_{7}}+e_{\alpha_{4}+\alpha_{7}},$$
$s_e$ has type $E_7(a_1)$,
$$
s_e=s_{\alpha_1}s_{\alpha_2} s_{\alpha_{3}+\alpha_{4}}s_{\alpha_{5}}s_{\alpha_{6}}s_{\alpha_{7}}s_{\alpha_{4}+\alpha_{7}};$$

\item
$E_8$, the type of $e$ is $E_8(a_1)$,
$$e=e_{\alpha_1} + e_{\alpha_2} + e_{\alpha_3} +e_{\alpha_{4}+\alpha_{5}}+e_{\alpha_{5}+\alpha_{8}}+e_{\alpha_{6}}+e_{\alpha_{7}}+e_{\alpha_{8}},$$
$s_e$ has type $E_8(a_1)$,
$$
s_e=s_{\alpha_1}s_{\alpha_2}s_{\alpha_3} s_{\alpha_{4}+\alpha_{5}}s_{\alpha_{5}+\alpha_{8}}s_{\alpha_{6}}s_{\alpha_{7}}s_{\alpha_{8}};$$

\item
$F_4$, the type of $e$ is $F_4(a_1)$,
$$e=e_{\alpha_1} + e_{\alpha_2}  +e_{\alpha_{2}+2\alpha_{3}}+e_{\alpha_{3}+\alpha_{4}},$$
$s_e$ has type $B_4$,
$$
s_e=s_{\alpha_1}s_{\alpha_2}s_{\alpha_{2}+2\alpha_{3}}s_{\alpha_{3}+\alpha_{4}};$$

\item
$G_2$, the type of $e$ is $G_2(a_1)$,
$$e=e_{2\alpha_1+\alpha_2} + e_{\alpha_2},$$
$s_e$ has type $A_2$,
$$
s_e=s_{3\alpha_{1}+\alpha_{2}}s_{\alpha_{2}}.$$

\end{itemize}

\end{lemma}

\begin{proof}
For the classical Lie algebras this lemma follows from the formula for the dimension of the centralizer $\z(e)$ of a nilpotent element $e$(see \cite{J}, Sect. 3.1--3.2). The values of the Kazhdan-Lusztig map for classical Lie algebras are calculated in \cite{SP}.

For the exceptional Lie algebras of types $E_6, E_7$ and $E_8$ the subregular nilpotent elements are also semiregular, and they are described explicitly in \cite{E}.
For the exceptional Lie algebras of types $F_4$ and $G_2$ the subregular nilpotent elements can be easily found using the tables of the dimensions of the
centralizers of nilpotent elements given in \cite{E} and the classification of nilpotent elements \cite{BK}. The values of the Kazhdan--Lusztig map for
exceptional Lie algebras were calculated in \cite{SP2}. The table of values given in \cite{SP2} is not complete. But in all cases one can find the value
of the Kazhdan--Lusztig map on the subregular nilpotent orbit. Note that for the semiregular elements the Kazhdan--Lusztig map has a very simple form,
$\sum_{\alpha\in \Delta}e_{\alpha}\mapsto \prod_{\alpha\in \Delta}s_{\alpha}$ (see \cite{C,CE}).

\end{proof}

Now using elements $s_e$ defined in Lemma \ref{subreg} and Proposition \ref{prop1} we construct transversal slices to conjugacy classes in the algebraic group $G$.
We start by fixing appropriate elements $h_i\in \h_i$, where $\h_i$ are the two--dimensional $s_e$--invariant subspaces arising in decomposition (\ref{hdec}), and $s_e$ acts on $\h_i$ as rotation with angle $\theta_i=\frac{2\pi}{m_i}\leq\pi$, $m_i\in \mathbb{N}$, or as the reflection with respect to the origin (which also can be regarded as rotation with angle $\pi$). We consider all cases listed in Lemma \ref{subreg} separately.

\begin{enumerate}

\item
$A_{r}$

Let $\h_{min}$ be the unique two--dimensional plane on which $s_e$ acts as rotation with the minimal among of all $\theta_i,i=1,\ldots,K$ angle $\theta_{min}=\frac{2\pi}{r}$. Order the terms in (\ref{hdec}) in such a way that $\h_{min}$ is labeled by the maximal index. One checks straightforwardly that $\overline{\Delta}_{min}=\Delta$ and that for any choice $h_{min}\in \h_{min}$ as in Section \ref{slices1} the length $l(s_e)$ of $s_e$ with respect to the corresponding system of simple positive roots is $r+1$, $l(s_e)=r+1$. Note that in this case $\h_0$ is one dimensional, $\h_0=\mathbb{R}\omega_r$, where $\omega_r$ is the fundamental weight corresponding to $\alpha_r$.

\item
$B_r$

Let $\h_{min}$ be the unique two--dimensional plane on which $s_e$ acts as rotation with the minimal among of all $\theta_i,i=1,\ldots,K$ angle $\theta_{min}=\frac{\pi}{r-1}$, and $\h_1=\mathbb{R}\alpha_r^\vee$. The Weyl group element $s_e$ acts on $\h_1$ as reflection with respect to the origin. Order the terms in (\ref{hdec}) in such a way that $\h_{min}$ is labeled by the maximal index. One checks straightforwardly that $\overline{\Delta}_{min}=\Delta\setminus \{\alpha_r,-\alpha_r\}$, $\overline{\Delta}_1=\{\alpha_r,-\alpha_r\}$, $\Delta=\overline{\Delta}_{min}\bigcup \overline{\Delta}_1$ and that for any choice $h_{min}\in \h_{min}$ and $h_1\in \h_1$ as in Section \ref{slices1} the length $l(s_e)$ of $s_e$ with respect to the corresponding system of simple positive roots is $r+2$, $l(s_e)=r+2$. Note that in this case $\h_0=0$.

\item
$C_r$

Let $\h_{min}$ be the unique two--dimensional plane on which $s_e$ acts as rotation with the minimal among of all $\theta_i,i=1,\ldots,K$ angle $\theta_{min}=\frac{\pi}{r-1}$, and $\h_1=\mathbb{R}\alpha_r^\vee$. The Weyl group element $s_e$ acts on $\h_1$ as reflection with respect to the origin. Order the terms in (\ref{hdec}) in such a way that $\h_{min}$ is labeled by the maximal index. One checks straightforwardly that $\overline{\Delta}_{min}=\Delta\setminus \{\alpha_r,-\alpha_r\}$, $\overline{\Delta}_1=\{\alpha_r,-\alpha_r\}$, $\Delta=\overline{\Delta}_{min}\bigcup \overline{\Delta}_1$ and that for any choice $h_{min}\in \h_{min}$ and $h_1\in \h_1$ as in Section \ref{slices1} the length $l(s_e)$ of $s_e$ with respect to the corresponding system of simple positive roots is $r+2$, $l(s_e)=r+2$. Note that in this case $\h_0=0$.

\item
$D_r$

A root subsystem $D_{r-2}+D_2\subset D_r$ is invariant under the action of $s_e$, and $s_e$ acts on $D_{r-2}$ $(D_2)$ as a Coxeter element of $D_{r-1}\supset D_{r-2}$ $(D_3\supset D_2)$, respectively, with respect to the natural inclusions.
Let $\h_{min}$ be the unique two--dimensional plane in the Cartan subalgebra corresponding to the root subsystem $D_{r-2}$ on which $s_e$ acts as rotation with the angle $\theta_{min}=\frac{\pi}{r-2}$, and let $\h_{1}$ be the unique two--dimensional plane in the Cartan subalgebra corresponding to the root subsystem $D_{2}$ on which $s_e$ acts as rotation with the angle $\theta_{1}=\frac{\pi}{2}$. Order the terms in (\ref{hdec}) in such a way that $\h_{min}$ is labeled by the maximal index. One checks straightforwardly that $\overline{\Delta}_{min}=\Delta\setminus D_2$, $\overline{\Delta}_1=D_2$, $\Delta=\overline{\Delta}_{min}\bigcup \overline{\Delta}_1$ and that for any choice $h_{min}\in \h_{min}$ and $h_1\in \h_1$ as in Section \ref{slices1} the length $l(s_e)$ of $s_e$ with respect to the corresponding system of simple positive roots is $r+2$, $l(s_e)=r+2$. Note that in this case $\h_0=0$.

\item
$E_6$, $E_7$, $E_8$, $G_2$

Let $\h_{min}$ be the unique two--dimensional plane on which $s_e$ acts as rotation with the minimal among of all $\theta_i,i=1,\ldots,K$ angle $\theta_{min}$, $\theta_{min}=\frac{2\pi}{9}$ for $E_6$, $\theta_{min}=\frac{\pi}{7}$ for $E_7$, $\theta_{min}=\frac{\pi}{12}$ for $E_8$, $\theta_{min}=\frac{\pi}{3}$ for $G_2$. Order the terms in (\ref{hdec}) in such a way that $\h_{min}$ is labeled by the maximal index. One checks straightforwardly that $\overline{\Delta}_{min}=\Delta$ and that for any choice $h_{min}\in \h_{min}$ as in Section \ref{slices1} the length $l(s_e)$ of $s_e$ with respect to the corresponding system of simple positive roots is $r+2$, $l(s_e)=r+2$. Note that in this case $\h_0=0$.

\item
$F_4$

The multiplicity of the eigenvalue $e^{\frac{\pi i}{3}}$ of $s_e$ with the minimal possible argument is $2$. Therefore there are infinitely many two--dimensional planes in $\h_{\mathbb{R}}$ on which $s_e$ acts as rotation with angle $\theta_{min}=\frac{\pi}{3}$. It turns out that one can choose a two--dimensional plane $\h_{min}\subset \h_{\mathbb{R}}$ on which $s_e$ acts as rotation with angle $\theta_{min}=\frac{\pi}{3}$ and which is not orthogonal to any root. Order the terms in (\ref{hdec}) in such a way that $\h_{min}$ is labeled by the maximal index. Then $\overline{\Delta}_{min}=\Delta$ and for any choice $h_{min}\in \h_{min}$ and $h_1\in \h_1$ as in Section \ref{slices1} the length $l(s_e)$ of $s_e$ with respect to the corresponding system of simple positive roots is $r+2$, $l(s_e)=r+2$. Note that in this case $\h_0=0$.

\end{enumerate}

\begin{remark}\label{br}
Note that in all cases 1-6 considered above the parabolic subalgebra $\p$ associated to $s_e$ with the help of the corresponding element $\bar{h}_0$ is a Borel subalgebra.
\end{remark}

The following proposition follows straightforwardly from Propositions \ref{prop2} and \ref{prop1}.

\begin{proposition}\label{subreg11}
Let $G$ be a complex simple algebraic group with Lie algebra $\g$, $H$ is a maximal torus of $G$, $\h$ the Cartan subalgebra in $\g$ corresponding to $H$.
Let $e$ be the subregular nilpotent element in $\g$ defined in the previous lemma and $s_e$ the element of the Weyl group of the pair $(\g,\h)$ associated to $e$ in Lemma \ref{subreg}. Denote by $s\in G$ the normal representative of $s_e$ in $G$ constructed in Section \ref{slices1}. Denote by $\b$ the Borel subalgebra of $\g$ associated to $s_e$ in Remark \ref{br}.
Let $B$ be the Borel subgroup of $G$ corresponding to the Borel subalgebra $\b$, $N$ the unipotent radical of $B$, $\overline B$ the opposite Borel subgroup in $G$, $\overline{N}$ the unipotent radical of $\overline B$.

Then the variety $N_sZs^{-1}$, where $Z=\{z\in H\mid szs^{-1}=z\}$, $N_s=\{n\in N\mid sns^{-1}\in \overline N\}$, is a transversal slice  to the set of conjugacy classes in $G$ and the conjugation map $N\times N_sZs^{-1}\rightarrow NZs^{-1}N$ is an isomorphism of varieties. The slice $N_sZs^{-1}$ has dimension $r+2$, $r=~{\rm rank}~\g$.
\end{proposition}

An immediate corollary of the last proposition and of Proposition \ref{sld} is the following construction of simple singularities in terms of the conjugation quotient map.

\begin{proposition}\label{subreg22}
Let $G$ be a complex simple algebraic group with Lie algebra $\g$, $H$ is a maximal torus of $G$, $\h$ the Cartan subalgebra in $\g$ corresponding to $H$.
Let $e$ be the subregular nilpotent element in $\g$ defined in the Lemma \ref{subreg} and $s_e$ the element of the Weyl group of the pair $(\g,\h)$ associated to $e$ in Lemma \ref{subreg}. Denote by $s\in G$ the representative of $s_e$ in $G$ constructed in Section \ref{slices1}. Let $\b$ be the Borel subalgebra of $\g$ associated to $s_e$ in Remark \ref{br}.
 Denote by $B$ the Borel subgroup of $G$ corresponding to the Borel subalgebra $\b$, and by $N$  the unipotent radical of $B$. Let $\overline B$ be the opposite Borel subgroup in $G$, $\overline{N}$ the unipotent radical of $\overline B$
 $Z=\{z\in H\mid szs^{-1}=z\}$ and $N_s=\{n\in N\mid sns^{-1}\in \overline N\}$.

Then fibers of the restriction of $\delta_G:G\rightarrow G/H$ to $N_sZs^{-1}$ are normal surfaces with isolated singularities. A point $x\in N_sZs^{-1}$ is an isolated singularity of such a fiber iff $x$
is subregular in $G$.

Moreover, if  $t\in H/W$, and $x\in N_sZs^{-1}$ is a singular point of the fiber $\delta_G^{-1}(t)$, then $x$ is a rational double point of type $_h\Delta_i$ for a suitable $i\in\{1,\ldots,m\}$, where $\Delta_i$ are the components in the decomposition of the Dynkin diagram $\Delta(t)$ of the centralizer $Z_G(t)$ of $t$ in $G$, $\Delta(t)=\Delta_1\cup\ldots \cup \Delta_m$. If $\Delta_i$ is of type $A$, $D$ or $E$ then $_h\Delta_i=\Delta_i$; otherwise $_h\Delta_i$ is the homogeneous diagram of type $A$, $D$ or $E$ associated to $\Delta_i$ by the rule $_hB_n=A_{2n-1}$, $_hC_n=D_{n+1}$, $_hF_4=E_6$, $_hG_2=D_4$.
\end{proposition}

\end{document}